\documentclass{article}
\usepackage{amssymb}
\usepackage{amsthm}
\usepackage{amsmath}
\usepackage{graphicx}
\usepackage{caption}
\usepackage{afterpage}
\usepackage{float}
\usepackage{lipsum}
\usepackage{subcaption}
\usepackage{amssymb}
\usepackage{mathrsfs}
\usepackage{hyperref}
\usepackage{geometry}
\usepackage{setspace}
\usepackage{cite}
\usepackage[utf8]{inputenc}
\usepackage[misc]{ifsym}
\usepackage{xcolor}
\usepackage{amsthm}
\usepackage{etoolbox}
\usepackage{indentfirst}
\usepackage{ifthen}
\usepackage[english]{babel}
\usepackage{csquotes}
\usepackage{lipsum}
\usepackage{mathtools}
\usepackage[export]{adjustbox}
\AtBeginEnvironment{proof}{%
}

\makeatletter
\renewcommand{\@makecaption}[2]{
	\vskip\abovecaptionskip
	\centering 
	\small
	\ifthenelse{\equal{#2}{}}
	{\textbf{\figurename~\thefigure.}}
	{\textbf{\figurename~\thefigure.}~#2}
	\vskip\belowcaptionskip
}

\makeatother

\newtheorem{theorem}{Theorem}[section] 
\newtheorem{definition}[theorem]{Definition} 
\newtheorem{lemma}[theorem]{Lemma} 
\newtheorem{corollary}[theorem]{Corollary} 
\newtheorem{proposition}[theorem]{Proposition} 

\title{The planar Tur\'an number of $\{K_{4},\Theta_{6}^{i}\}$}
\author{
	Jing Chen$^{1}$,
	Kai Gao$^{1}$,
	Yongxin Lan$^{2,}$\footnote{Corresponding author. E-mail: yxlan@hebut.edu.cn}, Changqing Xu$^{2}$,
	Shaoyu Zhao$^{2}$\\[2ex]
	\normalsize
	\begin{tabular}{@{}c@{}}
		1 School of Mathematics and Statistics, Shandong Normal University, Jinan, Shandong, China\\
		2 School of Science, Hebei University of Technology, Tianjin, China
	\end{tabular}
}
\date{}
\usepackage[numbers,sort&compress]{natbib}
\hypersetup{colorlinks=true,linkcolor=black,citecolor=black,urlcolor=black}
\allowdisplaybreaks

\begin{document}

	\maketitle
	\vspace{-2.5 cm}
	\vskip20pt \baselineskip12pt \vskip40pt\baselineskip16pt

	\noindent {\bf Abstract.}
	Let $\mathcal{H}$ be a family of graphs.
A graph is said to be $\mathcal{H}$-free if it contains no subgraph isomorphic to a graph in   $\mathcal{H}$.
The planar Tur\'an number $ex_{_\mathcal{P}}(n,\mathcal{H})$ is defined as the maximum number of edges in an $\mathcal{H}$-free planar graph on $n$ vertices.
In this paper, we determine the exact value of $ex_{_\mathcal{P}}(n,\{K_{4}, \Theta_{6}^{1}\})$ and a tight  upper bound of  $ex_{_\mathcal{P}}(n,\{K_{4}, \Theta_{6}^{2}\})$.
	
	  \section{Introduction}

	   In this paper, all graphs considered are simple, finite and undirected. For a graph \( G \), we denote by \( V(G) \) its vertex set, by \( E(G) \) its edge set, by \( v(G) \) the number of vertices and by \( e(G) \) the number of edges.
	    For a vertex \( v \in V(G) \), let \( d_G(v) \) denote its degree in \( G \) and let \( \delta(G) \) denote the minimum degree of \( G \). Let  \( C_n \) denote a cycle of length $n$.
	 For two  vertex‑disjoint graphs $H_1$ and $H_2$, we denote by $H_1+H_2$ the graph with vertex set $V(H_1)\cup V(H_2)$ and edge set $E(H_1)\cup E(H_2)\cup \{uv : u\in V(H_1),\, v\in V(H_2)\}$, and by $H_1\cup H_2$ the graph with  vertex set $V(H_1)\cup V(H_2)$ and edge set $E(H_1)\cup E(H_2)$.
	    A graph is \textit{planar} if it can be drawn in the plane so that its edges intersect only at their  endpoints.
	    A planar embedding of a planar graph is a \textit{plane graph}.
	    For a plane graph $G$, we use $\mathcal{F}(G)$ to denote the set of all faces of $G$.
	   The unbounded face of a plane graph is called an \textit{outer face}.
	   Every edge on the boundary of an outer face of $G$ is called an \textit{exterior edge} of $G$.
	   We use $\partial(F)$ to denote the set of edges on the boundary of a face $F$.
	    The \textit{degree} of a face $F$ is the number of edges on the boundary of $F$, where a cut edge is counted twice.
	    A plane graph is a \textit{plane triangulation}   if every face is of degree 3.
	    A \textit{decomposition} of a graph $G$ is a set of edge-disjoint subgraphs whose edge sets partition $E(G)$.

	   For a family of graphs $\mathcal{H}$, a graph is \textit{$\mathcal{H}$-free} if it does not contain any graph in $\mathcal{H}$ as a subgraph. When $\mathcal{H}=\{H\}$, we simply write $H$-free.
	    In 1941, Tur\'an \cite{Kr} determined the exact value of the Tur\'an number of $K_{\ell}$ (where $\ell$ is a positive integer with $\ell \geq 3$), that is, the maximum number of edges in a $K_{\ell}$-free graph on $n$ vertices, and its extremal graph is the balanced complete ($\ell-1$)-partite graph on $n$ vertices. This classical result in extremal graph theory has inspired  substantial   amount of related work, including the Erd\H{o}s-Stone theorem \cite{ES}, generalized Tur\'an problems \cite{Alon} and Tur\'an problems for hypergraphs \cite{hyper}.
	  In 2016, Dowden \cite{C45} considered Tur\'an-type~problems when the host graphs are planar graphs-specifically, how many edges can an $\mathcal{H}$-free planar graph on $n$ vertices have?
	  The  \textit{planar Tur\'an number} $ex_{_\mathcal{P}}(n,\mathcal{H})$ is the
	  maximum number of edges in an  $\mathcal{H}$-free  plane graph on $n$ vertices.
	   For cycle-free plane graphs, Dowden~\cite{C45} established  tight upper bounds of $ex_{_\mathcal{P}}(n,C_{4})$ and $ex_{_\mathcal{P}}(n,C_{5})$, while subsequent work \cite{C6,C7} determined  tight upper bounds of  $ex_{_\mathcal{P}}(n,C_{6})$ and $ex_{_\mathcal{P}}(n,C_{7})$.
	
	   \begin{figure}[H]
	   	\centering
	   	\includegraphics[width=0.4\linewidth,
	   	height=0.4\textheight,
	   	keepaspectratio,
	   	draft=false]{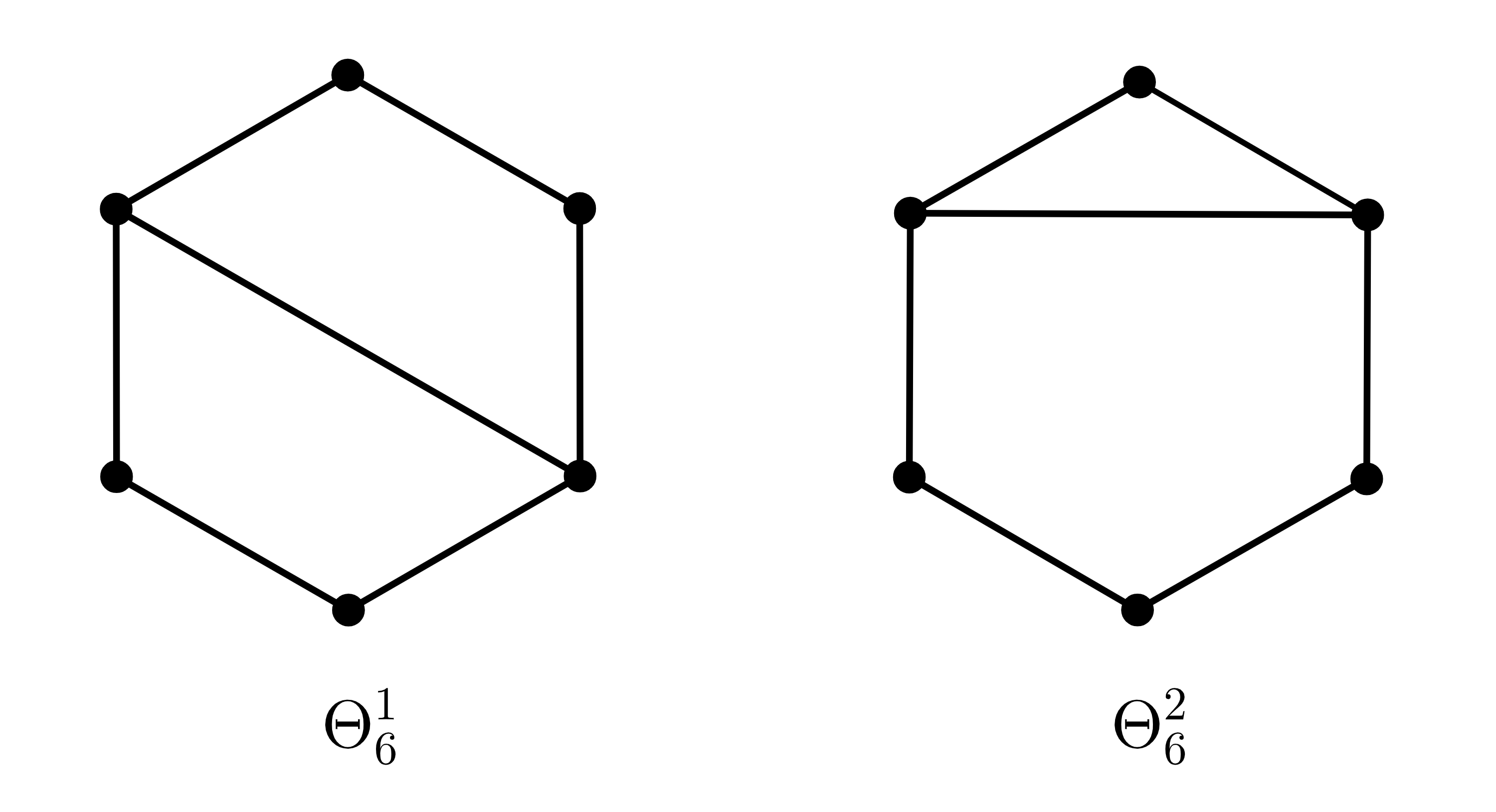}
	   	\caption {Two types of $\Theta_{6}$.}	
	   	\label{th6}
	   \end{figure}

	    \begin{theorem}\leavevmode
	   	\rm	\label{cycle}
	   	Let $n$ be an integer.

	   	\noindent	$(1)$ \cite{C45} For all $n\geq 3$,  $ex_{_\mathcal{P}}(n,C_{3})=2n-4$.
	   	
	   	\noindent	$(2)$ \cite{C45}  For all $n\geq 4$, $ex_{_\mathcal{P}}(n,C_{4})\leq \frac{15}{7}(n-2)$, with equality when $n\equiv30\ (\mathrm{mod}\ 70)$.
	
	   	\noindent	$(3)$ \cite{C45} For all $n\geq 11$, $ex_{_\mathcal{P}}(n,C_{5})\leq \frac{12}{5}n-\frac{33}{5}$,  with equality when $n\equiv9\ (\mathrm{mod}\ 15)$.

	   	\noindent	$(4)$ \cite{C6} For all $n\geq 18$, $ex_{_\mathcal{P}}(n,C_{6})\leq \frac{5}{2}n-7$, with equality when $n\equiv2\ (\mathrm{mod}\ 5)$.
	   	
	   	\noindent	$(5)$ \cite{C7}  For all $n\geq 39$, $ex_{_\mathcal{P}}(n,C_{7})\leq \frac{18}{7}n-\frac{48}{7}$ and the equality holds for infinitely many integers $n$.
	   \end{theorem}

	   The upper bound of $ex_{_\mathcal{P}}(n,\mathcal{H})$ has been studied in \cite{C345, C56, C67} when $\mathcal{H}$ contains two cycles.

	   \begin{theorem}\leavevmode\label{cij}
	  	\rm
	  	Let $n$ be an integer.
	  	
	  	\noindent	$(1)$ \cite{C345}  For all $n\geq 4$, $ex_{_\mathcal{P}}(n,\{C_{3},C_{4}\})\leq \frac{5}{3}(n-2)$, with equality when $n\equiv5\ (\mathrm{mod}\ 15)$.
	  	
	  	\noindent
	  	$(2)$ \cite{C345}  For all $n\geq 4$, $ex_{_\mathcal{P}}(n,\{C_{3},C_{5}\})=2n-4$.
	  	
	  		\noindent
	  	$(3)$ \cite{C345} For all $n\geq 8$, $ex_{_\mathcal{P}}(n,\{C_{4},C_{5}\})\leq 2n-6$, with equality when $n\equiv3\ (\mathrm{mod}\ 9)$.

	  	\noindent	$(4)$ \cite{C56} For all $n\geq 14$, $ex_{_\mathcal{P}}(n, \{C_{5},C_{6}\})\leq \frac{30}{13}n-\frac{84}{13}$, with equality  when $n\equiv7\ (\mathrm{mod}\ 10)$.
	  	
	  	\noindent	$(5)$ \cite{C67}  For all $n\geq 76$, $ex_{_\mathcal{P}}(n,\{C_{6},C_{7}\})\leq \frac{27}{11}n-\frac{72}{11}$, with equality when $n\equiv10\ (\mathrm{mod}\ 22)$.
	  	
	  \end{theorem}

	 Note that if a graph is $\{C_{3}, C_{k}\}$-free, then it is also $C_{3}$-free.
	 Thus
	 $ex_{_\mathcal{P}}(n,\{C_{3}, C_{k}\}) \leq ex_{_\mathcal{P}}(n,C_{3}) =2n-4$  for  $n \geq 4$.
	 The planar graph $K_{2,n-2}$ is  $\{C_{3}, C_{k}\}$-free for $k\geq 5$ and $n\geq4$.
	 Hence
	 $ ex_{_\mathcal{P}}(n,\{C_{3}, C_{k}\}) \geq 2n-4 $ for $n\geq4$ and  $k\geq 5$.
	 Then the following result extends the result of Theorem \ref{cij}(2).

	   \begin{theorem}\sl
	   	Let $n\geq 4$ and $k\geq 5$ be integers.
	   	Then
	    $ex_{_\mathcal{P}}(n,\{C_{3},C_{k}\})=2n-4$, with extremal graph $K_{2,n-2}$.
	   \end{theorem}

	   For $k\geq 4$, let
	 $\Theta_{k}$ be the family of \textit{Theta graphs} obtained from a cycle $C_{k}$  by adding an edge between two  non-adjacent vertices.
	Since $\Theta_k$ contains a single graph for $k\in\{4,5\}$, we will denote that unique graph by $\Theta_k$.
	   The set $\Theta_{6}$ contains two non-isomorphic graphs, denoted by $\Theta_{6}^{1}$ and $\Theta_{6}^{2}$, as shown in Figure \ref{th6}.
	    Lan et al. \cite{theta} studied upper bounds of $ex_{_\mathcal{P}}(n,\Theta_{4})$ and $ex_{_\mathcal{P}}(n,\Theta_{5})$.
	    Ghosh et al. \cite{theta6} provided  a tight upper bound of $ex_{_\mathcal{P}}(n,\Theta_{6})$, while Guan et al. \cite{theta 612} determined the bounds of $ex_{_\mathcal{P}}(n,\Theta_{6}^{1})$ and $ex_{_\mathcal{P}}(n,\Theta_{6}^{2})$.

	   \begin{theorem}
	   	\leavevmode
	  	\rm
	  	Let $n$ be an integer.

	  	\noindent	$(1)$ \cite{theta}  For all $n\geq 4$, $ex_{_\mathcal{P}}(n,\Theta _{4})\leq \frac{12}{5}(n-2)$, with equality when $n\equiv12\ (\mathrm{mod}\ 20)$.
	  	
	  	\noindent	$(2)$ \cite{theta}  For all $n\geq 5$, $ex_{_\mathcal{P}}(n,\Theta_{5})\leq \frac{5}{2}(n-2)$, with equality when $n\equiv50\ (\mathrm{mod}\ 120)$.
	  	
	  	\noindent	$(3)$ \cite{theta6}  For all $n\geq 14$, $ex_{_\mathcal{P}}(n,\Theta_{6})\leq \frac{18}{7}n-\frac{48}{7}$, with equality when $n\equiv24\ (\mathrm{mod}\ 54)$.
	  	
	  	\noindent	$(4)$ \cite{theta 612} For all $n\geq 6$, $ex_{_\mathcal{P}}(n,\Theta_{6}^{1})\leq \frac{45}{17}(n-2)$, with equality when $n\equiv70\ (\mathrm{mod}\ 170)$.
	  	
	  	\noindent	$(5)$ \cite{theta 612}  For all $n\geq 6$, $ex_{_\mathcal{P}}(n,\Theta_{6}^{2})\leq \frac{18}{7}(n-2)$, and $ex_{_\mathcal{P}}(n,\Theta_{6}^{2})\geq \frac{18}{7}n-\frac{48}{7}$ for infinitely many $n$.
	  \end{theorem}

	Gy{\H{o}}ri et al. \cite{C 7,K4C6}  studied the planar Turán number of $\{K_{4},C_{i}\}$ for $i\in\{5,6,7\}$. Fang \cite{K4theta5} gave a tight upper bound of $ex_{_{\mathcal{P}}}(n,\{K_{4},\Theta_{5}\})$.

	 \begin{theorem}
		\leavevmode
		\rm
		Let $n$ be an integer.
		
		\noindent	$(1)$ \cite{K4C6}  For all $n\geq 9$, $ex_{_\mathcal{P}}(n,\{K _{4},C_{5}\})\leq \frac{15}{7}(n-2)$, with equality when $n\equiv2\ (\mathrm{mod}\ 7)$.
		
		\noindent	$(2)$  \cite{K4C6} For all $n\geq 6$, $ex_{_\mathcal{P}}(n,\{K _{4},C_{6}\})\leq \frac{7}{3}(n-2)$, with equality when $n\equiv50\ (\mathrm{mod}\ 288)$.
		
			\noindent	$(3)$ \cite{C 7}  For all $n\geq60$, $ex_{_\mathcal{P}}(n,\{K _{4},C_{7}\})\leq \frac{18}{7}n-\frac{48}{7}$, with equality when $n\geq 110$ and $n\equiv26\ (\mathrm{mod}\ 42)$.

		\noindent	$(4)$  \cite{K4theta5}  For all $n\geq 25$, $ex_{_\mathcal{P}}(n,\{K _{4},\Theta_{5}\})\leq \frac{25}{11}(n-2)$, with equality when $n\equiv24\ (\mathrm{mod}\ 88)$.
	\end{theorem}

	In this paper, we determine the exact value of
	$ex_{_\mathcal{P}}(n,\{K_{4}, \Theta_{6}^{1}\})$ and a tight upper bound of $ex_{_\mathcal{P}}(n,\{K_{4}, \Theta_{6}^{2}\})$.

	  \begin{theorem}\label{k461}
	 	\sl
	 	Let $n$, $p\geq2$ and $0\leq m\leq1$ be integers such that $n=2p+2+m$.
	 	For all  $n \geq 6$,
	  $	ex_{_\mathcal{P}}(n,\{K_{4},\Theta_{6}^{1}\}) = \left\lfloor \frac{5}{2}(n-2) \right\rfloor,$ with extremal graph $P_{2}+(pK_{2}\cup mK_{1})$.
	 \end{theorem}

	 	 \begin{theorem}\label{k462}
	 	\sl
	 	Let $n$ be an integer.
	 	 For all $n \geq 6$, $ex_{_\mathcal{P}}(n,\{K_{4},\Theta_{6}^{2}\})\le \frac{12}{5}(n-2)$, with equality when $n\ge12$ and $n\equiv2\pmod{10}$.
	 \end{theorem}

	 Let $K_{5}^{-}$ be a graph obtained from $K_{5}$ by deleting an edge.
	 	If a graph is $\{K_{4},\Theta_{6}^{1},\Theta_{6}^{2}\}$-free, then it is also $\{K_{4}, \Theta_{6}^{2}\}$-free.
	 Thus
	 $ex_{_\mathcal{P}}(n,\{K_{4},\Theta_{6}^{1},\Theta_{6}^{2}\}) \leq ex_{_\mathcal{P}}(n,\{K_{4}, \Theta_{6}^{2}\}) \leq \frac{12}{5}(n-2)$  for  $n \geq 6$.
	 It is proved in \cite{theta} that  $ex_{_\mathcal{P}}(n,\{K_{5}^{-},\Theta_{6}^{1},\Theta_{6}^{2}\}) \leq \frac{12}{5}(n-2)$ for  $n\geq7$.
	 The extremal $\{K_{4},\Theta_{6}^{2}\}$-free  plane graphs  constructed in the proof of Theorem \ref{k462} are  $\{K_{4},\Theta_{6}^{1},\Theta_{6}^{2}\}$-free and also $\{K_{5}^{-},\Theta_{6}^{1},\Theta_{6}^{2}\}$-free.
	An immediate consequence of Theorem \ref{k462} is the following corollary.

	  \begin{corollary}\label{k46}
	 	\sl Let $n$ be an integer.
	 	
	 	\noindent
	 	$(1)$  For all  $n\geq6$, $ex_{_\mathcal{P}}(n,\{K_{4},\Theta_{6}^{1},\Theta_{6}^{2}\})\leq \frac{12}{5}(n-2)$, with equality when   $n\geq12$ and   $n\equiv2~(\text{mod}~10)$.

	 	\noindent
	 	$(2)$ For all  $n\geq7$, $ex_{_\mathcal{P}}(n,\{K_{5}^{-},\Theta_{6}^{1},\Theta_{6}^{2}\})\leq \frac{12}{5}(n-2)$, with equality when   $n\geq12$ and   $n\equiv2~(\text{mod}~10)$.
	 	
	 \end{corollary}

	\section{Preliminaries}
	
		In this section, we introduce some concepts that will be used throughout the paper.

		\begin{definition}[{\cite{theta 612}}]
			\label{eB}
			\rm
			Let $G$ be a plane graph. A \textit{triangular block} $B$ is a subgraph of $G$ constructed recursively as follows:
		 start with an edge $e\in E(G)$. If $e$ is not incident to any  face of degree $3$, then $B$ is induced by $e$;
		otherwise, we add $e$ to $E(B)$ and then repeatedly add all edges of every  face of degree 3 that shares an edge with $B$, until no further edges can be added.
		\end{definition}

	If $E(B)=\{e\}$, then $B$ is called a \textit{trivial triangular block}; otherwise, it is said to be nontrivial.
	An edge of $B$ is defined as \textit{interior edge} if it is incident to two faces of degree 3.
		Let  $\mathcal{B}(G)$ denote the collection of all triangular blocks of $G$.
		It is easy to see that $\mathcal{B}(G)$ forms a decomposition of 	$G$.
			For every triangular block $B$, define its edge contribution as $e(B)$. 	Consequently, the total number of edges of $G$ can be expressed as
	$$
	e(G)=\sum_{B\in \mathcal{B}(G)} e(B).
	$$

		\begin{definition}[{\cite{theta 612}}]
			\rm \label{fG}
		Let $e$ be an edge of a plane graph $G$ and let $F_{1}$ and $F_{2}$ denote the  two faces that are incident to $e$. If $e$ is incident to only one face, we set $F_{1}=F_{2}$.
		For any $F\in \mathcal{F}(G)$ and $e\in \partial(F)$, define $$f_{F}(e)=\frac{1}{|\partial(F)|}.$$
	The face contribution of $e$, denoted by $f(e)$, is defined as
$$
	f(e)=f_{F_{1}}(e)+f_{F_{2}}(e).
$$
	\end{definition}

	\begin{figure}[H]
		\centering
		\includegraphics[width=0.55\linewidth,
		height=0.4\textheight,
		keepaspectratio,
		draft=false]{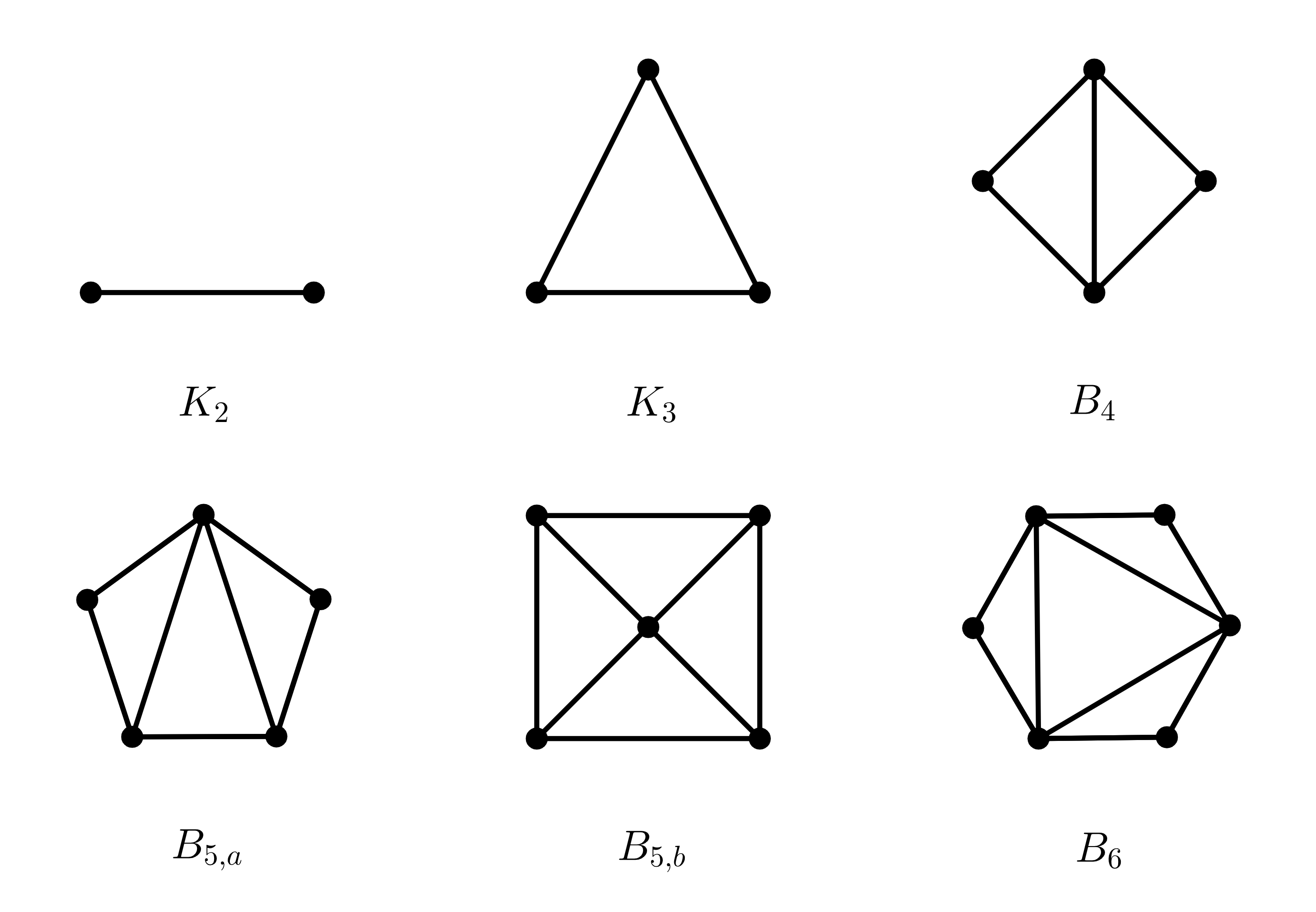}
		\caption {All $\{K_{4},\Theta_{6}^{1}\}$-free triangular blocks.}
		\label{5verticesblock}
	\end{figure} \noindent

	Let $B$ be a triangular block  of $G$. The face contribution of $B$, denoted by $f(B)$, is defined
\[
f(B) = \sum_{e \in B} f(e).
\]

	Let $f(G)$ be the total number of faces of $G$.
	Since $$\sum_{e \in \partial(F)} f_{F}(e) =\sum_{e \in \partial(F)}\frac{1}{|\partial(F)|}=|\partial(F)|\cdot
	\frac{1}{|\partial(F)|}=1,$$
	 we have
$$ \sum_{B \in \mathcal{B}(G)} f(B)=\sum_{B \in \mathcal{B}(G)} \sum_{e \in B} f(e)=\sum_{e \in E(G)} f(e)	=\sum_{e \in E(G)} (f_{F_{1}}(e)+f_{F_{2}}(e))	=\sum_{F \in \mathcal{F}(G)} \sum_{e \in \partial(F)} f_{F}(e)=\sum_{F \in \mathcal{F}(G)} 1=f(G).
$$

\vspace{0.1cm}

	\begin{definition} [{\cite{theta 612}}]
		\rm
		Let $T_1, T_2, \dots, T_k$ be a decomposition of $G$ such that $T_{i}$ consists of some triangular blocks of $\mathcal{B}(G)$ for each $i\in\{1,2, \dots, k\}$.
 Define
		$$
		e(T_{i}) = \sum_{\substack{B \in \mathcal{B}(G) \\ B \subseteq T_i}} e(B) \quad \text{and} \quad f(T_{i}) = \sum_{\substack{B \in \mathcal{B}(G) \\ B \subseteq T_i}} f(B).
		$$
	\end{definition}

 The following equations hold:
 $$\sum_{i=1}^{k}e(T_{i})=e(G)~\text{and}~ \sum_{i=1}^{k}f(T_{i})=\sum_{i=1}^{k}\sum_{\substack{B \in \mathcal{B}(G) \\ B \subseteq T_i}}f(B)=\sum_{B\in\mathcal{B}(G)}f(B)=f(G).$$

	  \vspace{0.1cm}
	 	\begin{proposition}[{\cite{theta 612}}]
	 	\label{61block}
	 	\rm
	 	If $G$ is a $\{K_{4}, \Theta_{6}^{i}\}$-free plane graph for $i\in \{1,2\}$, then $G$ has  at most six types of   triangular blocks  each of which is  as shown  in Figure \ref{5verticesblock}.
	 \end{proposition}

	\begin{lemma}\label{4no4}
		\rm
		If $G$ is a $\{K_{4}, \Theta_{6}^{1}\}$-free (or $\{K_{4}, \Theta_{6}^{2}\}$-free) plane graph on $n\geq6$ vertices with $\delta(G)\geq 3$, then
		the following results hold:
		
		\noindent	(1) No face of degree 4 shares exactly one common exterior edge with a subgraph  $B_{4}$.

		\noindent	(2) If every exterior edge of a triangular block $B_{4}$ is incident to a face of degree 4,  then the only possible configuration $B_{4}'$ is as shown in Figure~\ref{theta44}.
		
		\noindent(3)
		There exists at most one  face of degree 4 that is incident to some  exterior edges of  a triangular block $B_{5,a}$ and if it exists, the
		 only possible configuration is as shown in Figure~\ref{theta514}$(c)$.
		
		\noindent(4)
		Every exterior edge of a triangular block $B_{5,b}$ is incident to a face of degree  at least $5$.
	\end{lemma}

	\begin{proof}[\bf Proof]
	\noindent(1) Suppose that $F$ is a face of degree 4 that shares exactly one common exterior edge with a subgraph $B_{4}$.
 Since $\delta(G)\geq3$, the boundary of $F$ is a cycle.
	Since $G$ is a simple plane graph on $n\geq6$ vertices, $F$  shares at most three common vertices with $B_{4}$.
	All possible configurations are shown in Figure \ref{Fig:the44}  where the blue line indicates the common edge of $F$ and $B_{4}$ and the shaded region may contain some vertices.
	If $F$ shares two common
	vertices with $B_{4}$ as shown in Figure \ref{Fig:the44}$(a)$, then $G$ contains $\Theta_{6}^{1}$ and $\Theta_{6}^{2}$ as subgraphs, a contradiction.
	If $F$ shares three common vertices with $B_{4}$, the only possible configuration is as shown in  Figure \ref{Fig:the44}$(b)$.
	Because $F$ shares exactly one exterior edge with $B_{4}$,  $G$ contains $K_{4}$ as a subgraph, a contradiction.
		\vspace{-0.3cm}
			\begin{figure}[H]
			\centering
			\includegraphics[width=0.38\linewidth,
			height=0.4\textheight,
			keepaspectratio,
			draft=false]{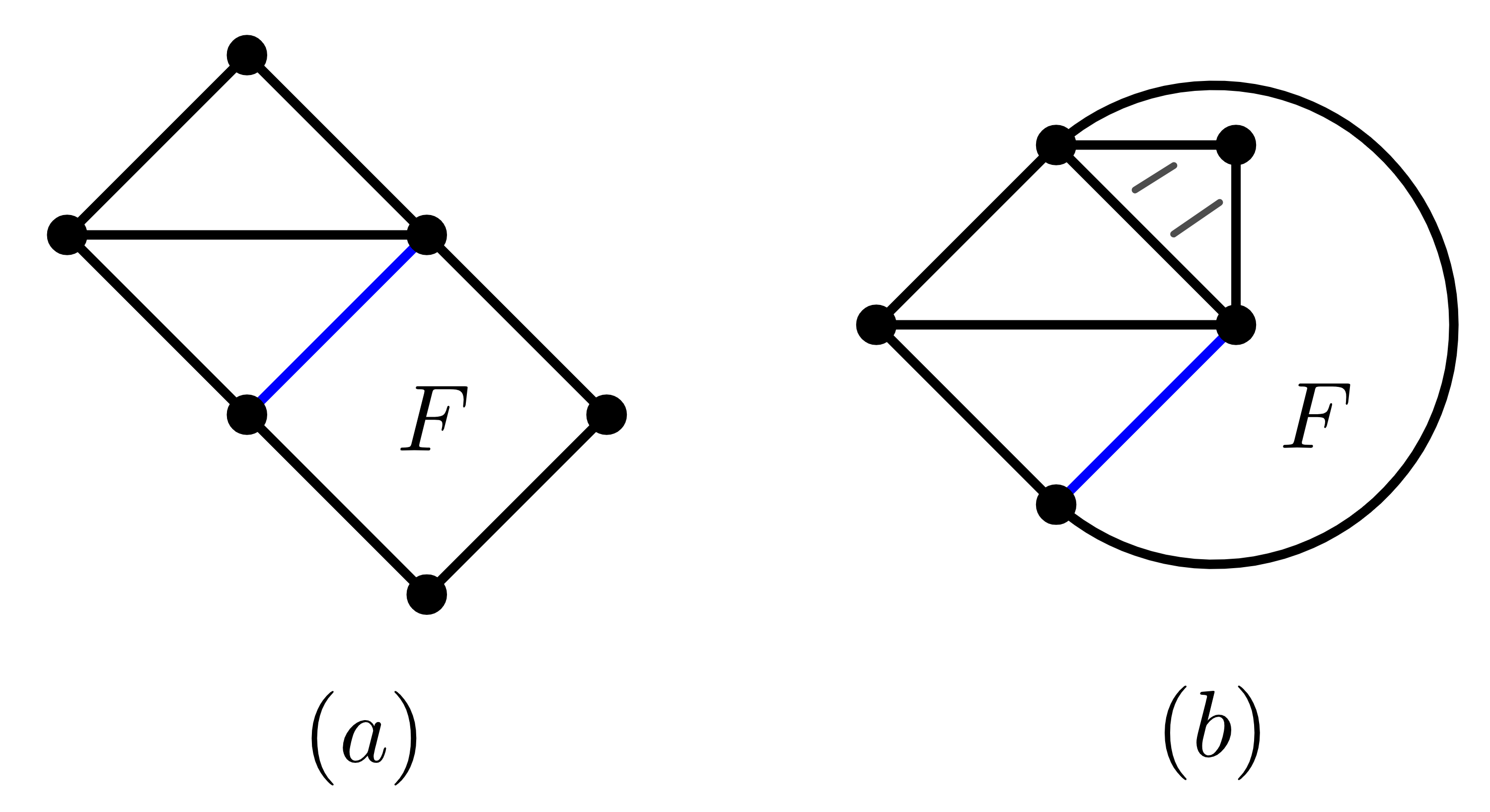}
				\caption{A face $F$ of degree 4 shares one common edge with $B_{4}$.}
			\label{Fig:the44}
		\end{figure}
			\vspace{-0.4cm}
		\noindent(2)
		Let $F$ be a face of degree 4 that is incident to some exterior edges of $B_{4}$.
	Since $n\geq6$,	by the result (1), $F$ shares exactly two exterior edges with $B_{4}$.
		Since $\delta(G)\geq 3$,  the only possible configuration is $B_{4}'$ as shown in Figure~\ref{theta44}.\vspace{-0.4cm}
		\begin{figure}[H]
			\centering
			\includegraphics[width=0.3\linewidth,
			height=0.4\textheight,
			keepaspectratio,
			draft=false]{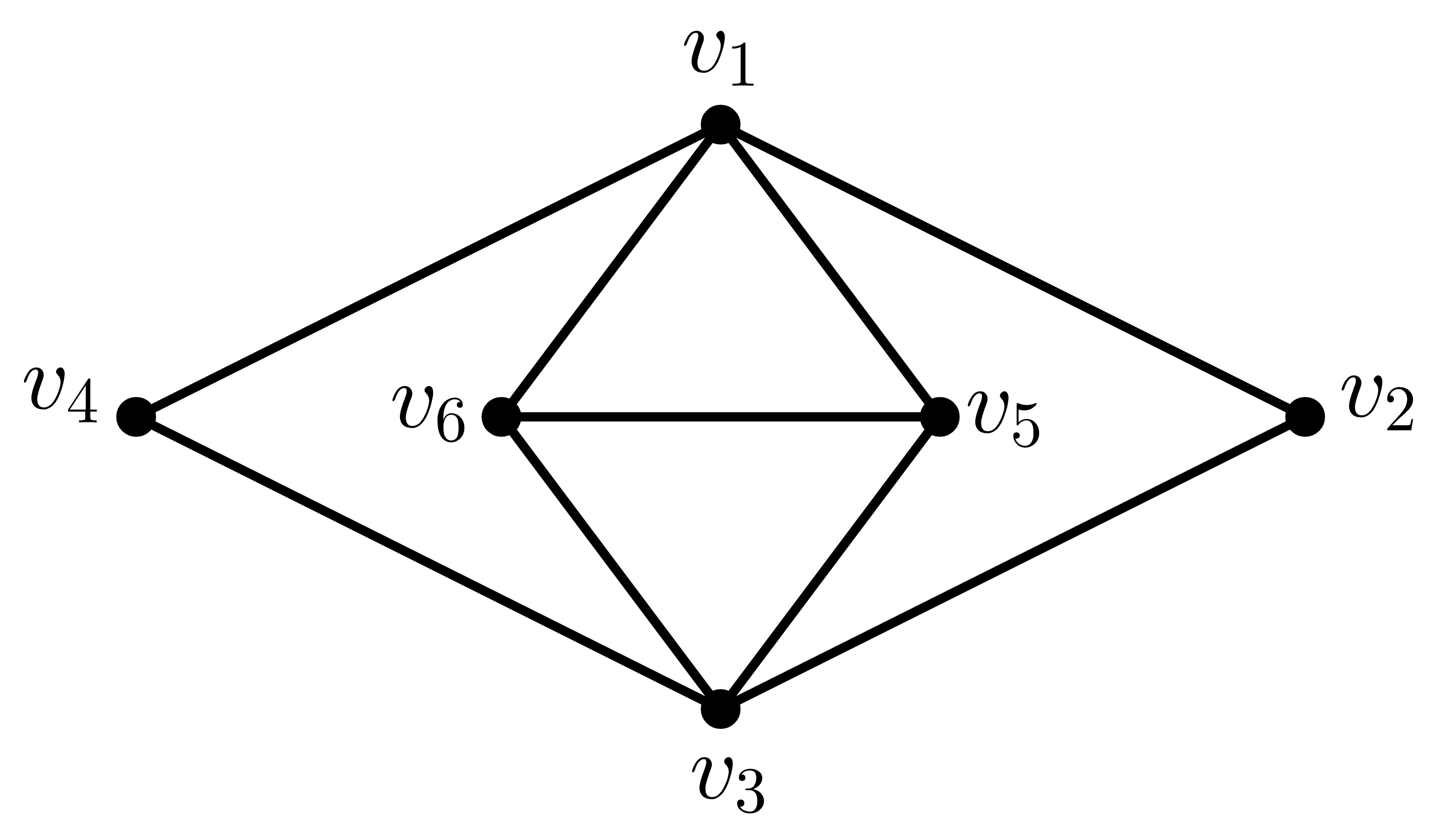}
		\caption{A graph  $B_{4}'$.}
		\label{theta44}
		\end{figure}
		\vspace{-0.4cm}
		\noindent(3)
		Let $B_{5,a}$ be a triangular block with vertex set $\{v_{1},v_{2},v_{3},v_{4},v_{5}\}$ as shown in Figure \ref{theta514}.
		Since $G$ is $\Theta_{6}^{1}$-free or $\Theta_{6}^{2}$-free and every  exterior edge of $B_{5,a}$ is on a subgraph $B_{4}$, we see from the result (1)  that every face of degree 4 shares at least two exterior edges with  $B_{5,a}$.
		Let $F$ be a face of degree 4.
		If $F$ shares exactly two exterior edges with $B_{5,a}$, then all possibilities are  shown in Figure \ref{theta514}$(a)$ and $(b)$.
		In Figure \ref{theta514}$(a)$, $G$ contains a subgraph $\Theta_{6}^{1}$ consisting of two cycles $v_{2}v_{6}v_{5}v_{1}v_{2}$ and $v_{2}v_{3}v_{4}v_{1}v_{2}$, and a subgraph $\Theta_{6}^{2}$ consisting of two cycles $v_{3}v_{4}v_{1}v_{3}$ and $v_{3}v_{1}v_{5}v_{6}v_{2}v_{3}$, a contradiction.
		In Figure \ref{theta514}$(b)$, $G$ contains a subgraph $\Theta_{6}^{1}$ consisting of two cycles $v_{2}v_{1}v_{4}v_{3}v_{2}$ and $v_{3}v_{4}v_{5}v_{6}v_{3}$, and a subgraph $\Theta_{6}^{2}$ consisting of two cycles $v_{1}v_{2}v_{3}v_{1}$ and $v_{1}v_{4}v_{5}v_{6}v_{3}v_{1}$, a contradiction.
		Therefore,  $F$ shares at least three exterior edges with $B_{5,a}$. Since $\delta(G)\geq 3$, the only possible configuration is as shown in  Figure \ref{theta514}(c).
		Hence, there exists at most one face of degree 4 that is incident to some exterior edges of $B_{5,a}$.
		\vspace{-0.28cm}
		\begin{figure}[H]
			\centering
			\includegraphics[width=0.5\linewidth,
			height=0.4\textheight,
			keepaspectratio,
			draft=false]{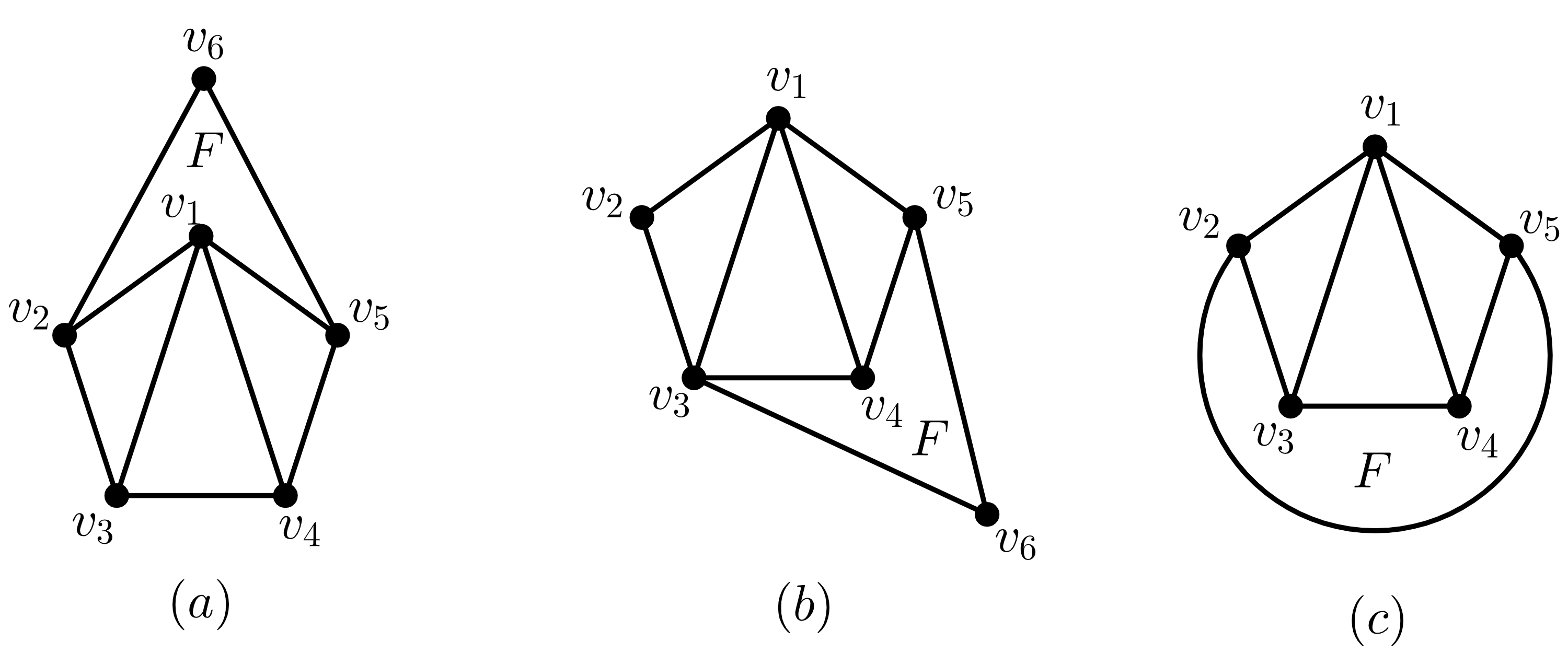}
				\caption{Two or three exterior edges of $B_{5,a}$ are incident to a face of degree 4.}
			\label{theta514}
		\end{figure}
			\vspace{-0.3cm}
		\noindent(4)
		Let $B_{5,b}$ be a triangular block with vertex set $\{v_{1},v_{2},v_{3},v_{4},w\}$  as shown in Figure~\ref{figB5b}.
		Let $e_i = v_i v_{i+1}$ where $i \in \{1,2,3,4\}$ and $i+1$ is taken modulo $4$.
		Then $e_{i}$ is incident to a face with boundary $v_{i}v_{i+1}wv_{i}$.
		Let $F_{i}$ be another  face that is incident to  $e_{i}$.
			According to  the definition of  a triangular block, $F_{i}$ is of degree at least 4.
		If  $F_{i}$ is of  degree 4, then by the result (1) and $n\geq6$, $F_{i}$ shares exactly two common edges with $B_{5,b}$.
		Suppose that $e_{i},e_{i+1}\in E(F_{i})$.
		Note that $e_{i}$ is an exterior edge of a subgraph $B_{4}$ with vertex set  $\{v_i, v_{i+1},w,v_{i+3}\}$.
		Therefore, $F_{i}$ shares exactly one common edge
		$e_{i}$ with this subgraph $B_{4}$, which  contradicts the result (1).
		Hence  every exterior edge of $B_{5,b}$ is incident to a face of degree  at least $5$.
	\end{proof}
	\vspace{-0.3cm}
		\begin{figure}[H]
		\centering
		\includegraphics[width=0.18\linewidth,
		height=0.4\textheight,
		keepaspectratio,
		draft=false]{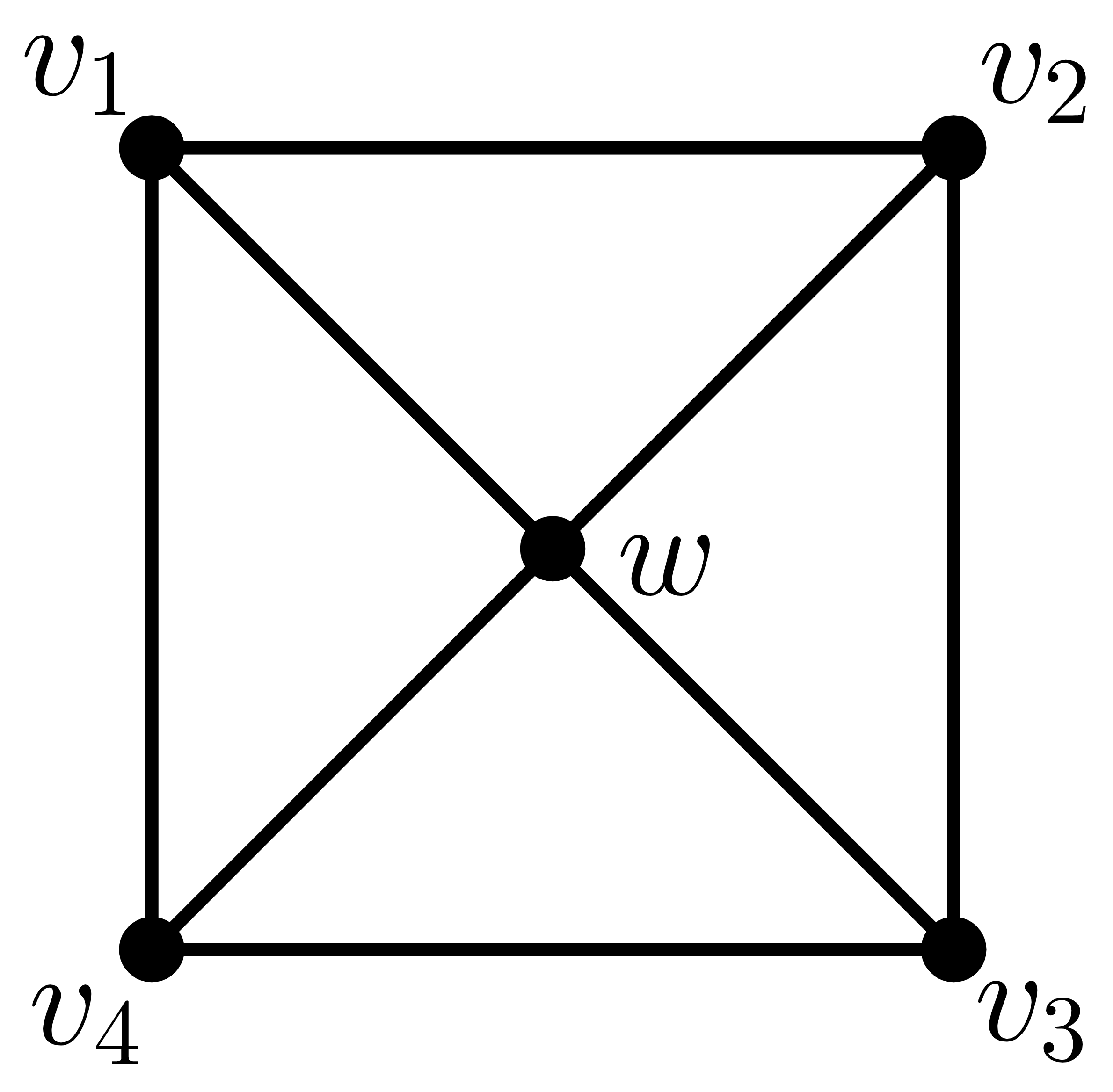}
	\caption{A triangular block $B_{5,b}$}
	\label{figB5b}
	\end{figure}

	\section{Planar Tur\'an number   of  $\{K_{4},\Theta_{6}^{1}\}$}

	In this section, we determine
	$ex_{_\mathcal{P}}(n,\{K_{4},\Theta_{6}^{1}\})= \left\lfloor \frac{5}{2}(n-2) \right\rfloor$ for $n\geq 6$, and
	 construct an infinite family of extremal $\{K_{4},\Theta_{6}^{1}\}$-free plane  graphs.

	\begin{lemma}\label{b6f5}
		\rm
		If $G$ is a $\{K_{4},\Theta_{6}^{1}\}$-free plane graph on $n\geq6$ vertices with $\delta(G)\geq3$, then every exterior edge of a triangular block $B_{6}$ is incident to a  face of degree at least 5.
	\end{lemma}
	
	\begin{proof}[\bf Proof]
		Let $B_{6}$ be a triangular block with vertex set $\{v_{1},v_{2},v_{3},v_{4},v_{5},v_{6}\}$ as shown in Figure \ref{Fig:b6f5}.
		According to  the definition of  a triangular block, every exterior edge of $B_{6}$ is not incident to a  face of degree 3.
		Suppose that there exists a   face $F$ that is of degree 4 and incident to some exterior edges of $B_{6}$.
		Since $\delta(G)\geq3$, $F$ can share at most two common edges with  $B_{6}$.
		Since every exterior edge of $B_{6}$ is contained in a subgraph that is isomorphic to $B_{4}$, by Lemma \ref{4no4}(1), $F$ shares exactly two  common edges with $B_{6}$.
		Then $F$ shares three or four common vertices with $B_{6}$.
		If $F$ shares three common vertices with $B_{6}$, then since $\delta(G)\geq 3$, the only possible configuration is as shown in Figure~\ref{Fig:b6f5}$(a)$, where $F$ is induced by $\{v_{1},v_{2},v_{3},v_{7}\}$.
		If $F$ shares four common vertices with $B_{6}$, then the only possible configuration is as shown in Figure~\ref{Fig:b6f5}$(b)$, where $F$ is induced by $\{v_{1},v_{3},v_{4},v_{5}\}$ and the shaded region may contain some vertices.
		In 	Figure \ref{Fig:b6f5}$(a)$, $G$ contains a subgraph $\Theta_{6}^{1}$ consisting of two cycles  $v_{1}v_{2}v_{3}v_{7}v_{1}$ and $v_{2}v_{3}v_{4}v_{6}v_{2}$, a contradiction.
		In 	Figure \ref{Fig:b6f5}$(b)$, $G$ contains a subgraph $\Theta_{6}^{1}$ consisting of two cycles  $v_{2}v_{4}v_{5}v_{6}v_{2}$ and $v_{1}v_{3}v_{4}v_{5}v_{1}$, a contradiction.
		Therefore every exterior edge of a triangular block $B_{6}$ is incident to a  face of degree at least 5.
	\end{proof}

		\begin{figure}[H]
		\centering
		\includegraphics[width=0.45\linewidth,
		height=0.4\textheight,
		keepaspectratio,
		draft=false]{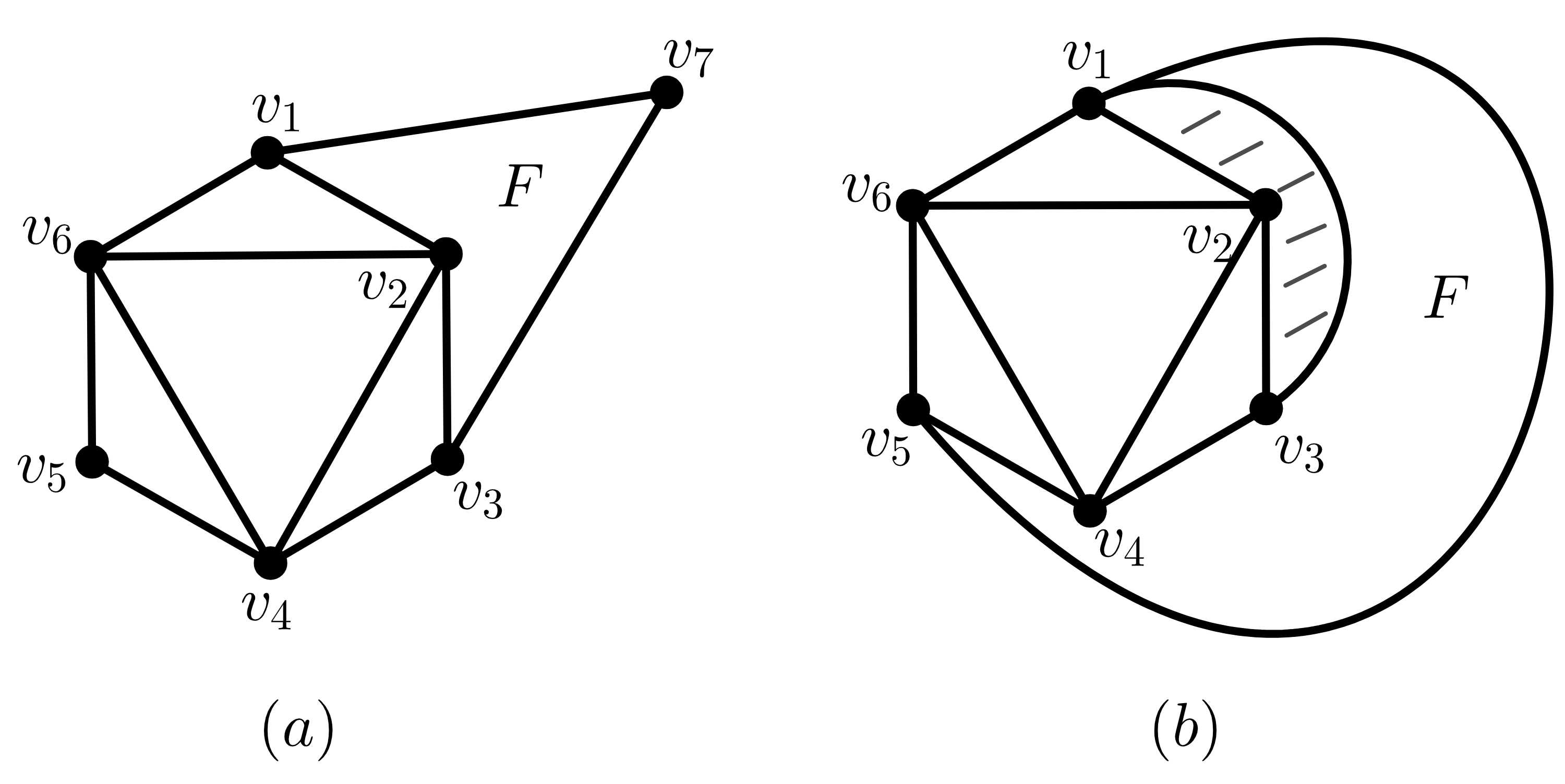}
		\caption {A face of degree 4 shares exactly two common edges with $B_{6}$.}	
		\label{Fig:b6f5}
	\end{figure}

	\begin{lemma}\label{K4theta61}
		\rm
		If $G$ is a $\{K_{4},\Theta_{6}^{1}\}$-free plane graph on $n\geq 6$ vertices  with $\delta(G)\geq 3$, then
		$e(G)\leq \frac{5}{2}(n-2)$.
	\end{lemma}
	
	\begin{proof}[\bf Proof]
		By Euler's formula, the inequality  $e(G)\leq \frac{5}{2}(n-2)$ is equivalent to $5f(G)-3e(G)\leq 0$.
		Since $5f(G)-3e(G)=\sum_{B\in \mathcal{B}(G)}(5f(B)-3e(B))$, it suffices to show that $5f(B)-3e(B)\leq 0$ for every  triangular block $B$ of $G$.
		By Proposition \ref{61block}, $G$ has six types of triangular blocks $K_{2}$, $K_{3}$, $B_{4}$, $B_{5,a}$, $B_{5,b}$ and $B_{6}$.

		\vskip3pt\noindent  {\bf Case 1.}
		$B=K_{2}$.
		
		According to  the definition of  a triangular block, this edge is incident to two faces each of which is  of  degree at least 4.
		Then
		$$5f(B)-3e(B)\leq 5\biggl(\frac{1}{4}+\frac{1}{4}\biggl)-3=-\frac{1}{2}<0.$$

		\vskip3pt\noindent  {\bf Case 2.}
		$B=K_{3}$.
		
			According to  the definition of  a triangular block, every edge of $B$ is incident to a face of degree at least 4.
		Hence  $$5f(B)-3e(B)\leq 5\biggl(1+\frac{3}{4}\biggl)-3\times 3=-\frac{1}{4}<0.$$

		\vskip3pt\noindent  {\bf Case 3.}
		$B=B_{4}$.
		
		By Lemma \ref{4no4}(2),	it is possible that every edge of $B$ is incident to a face of degree 4.
		Thus			$$5f(B)-3e(B)\leq 5\biggl(2+\frac{4}{4}\biggl)-3\times 5=0.$$

		\vskip3pt\noindent  {\bf Case 4.}
		$B=B_{5,a}$.

		If every exterior edge of $B$ is incident to a face of degree at least 5, then  $$5f(B)-3e(B)\leq5 \biggl(3+\frac{5}{5}\biggl)-3\times 7=-1<0.$$

		If  $B$ has some exterior edges that are incident to a face of degree 4, then by Lemma \ref{4no4}(3), three exterior edges of $B$ are  incident to a face of degree 4 and the other two exterior edges are incident to a face of degree  at least 5.
		Thus  $$5f(B)-3e(B)\leq5 \biggl(3+\frac{3}{4}+\frac{2}{5}\biggl)-3\times 7=-\frac{1}{4}<0.$$
		
		\vskip3pt\noindent  {\bf Case 5.}
		$B=B_{5,b}$.

		Lemma \ref{4no4}(4) implies that $f(B)\leq 4+\frac{4}{5}$.
		Hence $$5f(B)-3e(B)\leq5 \biggl(4+\frac{4}{5}\biggl)-3\times 8=0.$$
		
		\vskip3pt\noindent  {\bf Case 6.}
		$B=B_{6}$.

		By Lemma \ref{b6f5}, every exterior edges of $B_{6}$ is incident to a face of degree at least 5.
		Then  $$5f(B)-3e(B)\leq5 \biggl(4+\frac{6}{5}\biggl)-3\times 9=-1<0.$$
		
		For every triangular block $B$ of $G$, we have  $5f(B)-3e(B)\leq 0$. Hence  $e(G)\leq \frac{5}{2}(n-2)$, as desired.
	\end{proof}
	\vspace{0.2cm}
	
	\noindent{\bf Proof of Theorem \ref{k461}.}
	Let $G$ be an  extremal $\{K_{4},\Theta_{6}^{1}\}$-free plane graph on $n\geq 6$ vertices.
	We first prove that $e(G)\leq \frac{5}{2}(n-2)$, and so  $e(G)\leq \left\lfloor \frac{5}{2}(n-2) \right\rfloor$.
	The inequality $e(G)\leq \frac{5}{2}(n-2)$ is equivalent to $5n-2e(G)\geq 10$.
	We repeatedly do the operation of deleting vertices of degree at most 2 and
	get an induced subgraph $G'$.  If $G'$ is not an
	empty graph, then $\delta(G')\geq 3$.
	Let $e_{i}$ be the number of edges deleted when deleting the $i$-th vertex where $i=1,2,\ldots,n-v(G').$
	It is clear that  $e_{i}\leq 2$ and $5-2e_{i}\geq 1$.
	Therefore
	\begin{align*}
		5n-2e(G)&=5(v(G')+(n-v(G')))-2(e(G')+e_{1}+e_{2}+\dots+e_{n-v(G')})\\
		&= 5v(G')-2e(G')+\sum_{i=1}^{n-v(G')}(5-2e_{i}) \\
		&\geq 5v(G')-2e(G')+(n-v(G')).
	\end{align*}
	If $n - v(G') = 0$, then $\delta(G) \geq 3$.  By Lemma \ref{K4theta61}, $5n - 2e(G) \geq 10$.
	We next consider the case that $n - v(G') \geq  1$.
	
	If $v(G')=0$, then since  there exists at most one edge  between the last two vertices and $n\geq 6$, we have $$5n-2e(G)\geq 5\times 2-2\times 1+(n-2)\geq12.$$

	If $v(G')\neq 0$, let  $G_{1}', G_{2}', \ldots ,G_{k}'$ be all components of $G'$ where $k\geq 1$.
	Since $\delta(G')\geq 3$ and $G'$ is $K_{4}$-free, we have $v(G_{i}')\geq 5$ for every $i\in\{ 1,2,\ldots,k\}$.
	If $v(G_{i}')=5$, then since
	a plane triangulation on five vertices has  a subgraph isomorphic to $K_{4}$  with nine edges, we have $e(G_{i}')\leq 8$ and so $5v(G_{i}')-2e(G_{i}')\geq 5\times 5-2\times8 =9$.
	If $v(G_{i}')\geq 6$, then by Lemma \ref{K4theta61}, we have $5v(G_{i}')-2e(G_{i}')\geq 10$.
	Therefore  $$5n-2e(G)\geq \sum_{i=1}^{k}(5v(G_{i}')-2e(G_{i}'))+(n-v(G'))\geq 9k+(n-v(G'))\geq 10.$$	
	Therefore $e(G)\leq \frac{5}{2}(n-2)$  and so  $e(G)\leq \left\lfloor \frac{5}{2}(n-2) \right\rfloor$ for all integer $n\geq6$.

We next prove that
$e(G) \geq \left\lfloor \frac{5}{2}(n-2) \right\rfloor$ by constructing a  $\{K_{4},\Theta_{6}^{1}\}$-free plane graph achieving this bound for every integer $n \geq 6$.
	Let $p\geq 2$ and $0\leq m\leq 1$ be integers satisfying $n = 2(p+1) + m$.
	Let $H = 2K_1 + (pK_2 \cup mK_1)$. It is clear that $e(H) = 5p + 2m = \left\lfloor \frac{5}{2}(n-2) \right\rfloor$.
		Since the number of common edges of any two cycles of length 4 in $H$ is either 0 or 2, $H$ is $\Theta_6^1$-free.
Moreover, for every vertex of degree at least 3 in $H$,  the subgraph induced by its neighbors in $H$ is \(C_3\)-free. Hence, $H$ is \(K_4\)-free.
	
		Therefore $e(H) = \left\lfloor \frac{5}{2}(n-2) \right\rfloor$ for  $n\geq6$.
	This completes the proof of Theorem \ref{k461}.
	$\hfill\blacksquare$
	\vskip3pt

		\section{Planar Tur\'an number of   $\{K_{4},\Theta_{6}^{2}\}$}
	
	 In this section, we determine
	 $ex_{_\mathcal{P}}(n,\{K_{4},\Theta_{6}^{2}\})\leq \frac{12}{5}(n-2)$ for  $n\geq 6$ and
	 construct  two types of extremal $\{K_{4},\Theta_{6}^{2}\}$-free plane  graphs.

	 \begin{lemma}\label{L:k4the62s}
	 	\rm
	 	If $G$ is a $\{K_{4},\Theta_{6}^{2}\}$-free plane graph on $n\geq 6$ vertices   with $\delta(G)\geq 3$, then the following results hold:

	 	\noindent(1) A cycle of length 3  cannot share exactly  two vertices that are adjacent on a cycle  of length 5.

	 	\noindent(2) Every  exterior edge of $B_{4}'$ is a trivial triangular block.

	 	\noindent(3) If there exists a face of degree 4 that is incident to some  exterior edges of  a triangular block  $B_{5,a}$,   then each of the remaining exterior edges of $B_{5,a}$ is incident to a face of degree at least 6.
	 	
	 	\noindent(4) Every exterior edge of a triangular block $B_{5,b}$ is incident to a face of degree at least 6.
	 	
	 \end{lemma}

	 \begin{proof}[\bf Proof]
	 	
	 	\noindent(1) Since $G$ is $\Theta_{6}^{2}$-free,  the result holds obviously.

	 	\noindent(2)
	 	Let $B_{4}'$ be a graph with vertex set $\{v_{1},v_{2},v_{3},v_{4},v_{5},v_{6}\}$ as shown in Figure \ref{theta44}.
	 	Assume that there exists an exterior edge, say $v_{1}v_{2}$, that is incident to a face $F$ of degree 3.
	 	Note that $v_{1}v_{2}$ lies on a cycle $C=v_{1}v_{6}v_{5}v_{3}v_{2}v_{1}$.
	 	Since $\delta(G)\geq 3$, the third vertex of $F$ is not on the cycle $C$, which contradicts the result $(1)$.

	 	\noindent(3) By Lemma \ref{4no4}(3), the only possible configuration is as shown in Figure \ref{theta514}(c).
	 	We can see that the edges $v_1v_2$ and $v_1v_5$ lie on faces of degree 3 with boundaries $C_1 = v_1v_2v_3v_1$ and $C_2 = v_1v_4v_5v_1$, respectively.
	 	Suppose that  the edge $v_1v_2$  or $v_1v_5$ is incident to a face of degree 5.
	 	Since neither $v_{3}$ nor $v_{4}$ is on the boundary of this face,   $C_{1}$  or $C_{2}$ must share exactly two vertices that are adjacent on the boundary of this face, which contradicts the result (1).

	 	\noindent(4) Let $B_{5,b}$ be a triangular block with vertex set $\{v_1,v_2,v_3,v_4,w\}$ as shown in Figure \ref{figB5b}. By Lemma \ref{4no4}(4), every exterior edge of $B_{5,b}$ is incident to a face of degree at least 5.
	 	In Figure \ref{figB5b}, we see that every exterior edge $e_i = v_i v_{i+1}$ of  $B_{5,b}$ lies on a face $F_i$ of degree 3 with boundary $C(F_i) = v_i v_{i+1} w v_i$, where $i \in \{1,2,3,4\}$ and $i+1$ is taken modulo 4.
	 	If there exists a face $F$ of degree 5 that is incident to an exterior edge $e_{i}$ of $B_{5,b}$, then $F$ shares exactly two vertices  with $F_{i}$ because $w\notin\partial(F)$,  which contradicts the result (1).
	 \end{proof}
	
	 	\begin{lemma}\label{L:K4theta62}
	 	\rm
	 	If $G$ is a $\{K_{4},\Theta_{6}^{2}\}$-free plane graph on $n\geq 6$ vertices  with $\delta(G)\geq 3$, then
	 	$e(G)\leq \frac{12}{5}(n-2)$.
	 \end{lemma}
	
	 \begin{proof}[\bf Proof]
	
	By Euler's formula, the inequality $e(G) \le \frac{12}{5}(n-2)$ can be rewritten as $12f(G) - 7e(G) \le 0$.
	Let $T_1, T_2, \ldots, T_k$ be a decomposition of  $G$. Since $$12f(G) - 7e(G) = 12 \sum_{i=1}^{k} f(T_i) - 7\sum_{i=1}^{k} e(T_i),$$   it is enough to prove that $$12f(T_i) - 7e(T_i) \leq 0 $$ for each $i \in \{1, 2, \ldots, k\}$.
	We first give an upper bound for $12f(B) - 7e(B)$ for each type of triangular block $B$ that may possibly exist in $\mathcal{B}(G)$.
	
	\vskip3pt\noindent  {\bf Case 1.}
	$B=K_{2}$.
	
	By the definition of a triangular block, this edge is incident to two faces, each of which is of degree at least 4. Then
	$$12f(B)-7e(B)\leq 12\biggl(\frac{1}{4}+\frac{1}{4}\biggl)-7=-1<0.$$

	\vskip3pt\noindent  {\bf Case 2.}
	$B=K_{3}$.
	
	By the definition of a triangular block, every edge of $B$ is incident to a face of degree at least 4.
	So $$12f(B)-7e(B)\leq 12\biggl(1+\frac{3}{4}\biggl)-7\times 3=0.$$
	
	\vskip3pt\noindent  {\bf Case 3.}
	$B=B_{4}$.

	If  a face $F$  of degree 4  is incident to some exterior edges of $B$, then because $\delta(G)\geq3$  and $n\geq6$, $F$ shares at most two common exterior edges with $B$.
	By Lemma \ref{4no4}(1), $F$ shares exactly two common exterior edges with $B$.
	Hence either two or four exterior edges of $B$ are incident to  faces of degree 4.
	
	If exactly four exterior edges  of $B$ are incident to  faces of degree 4, then $$12f(B)-7e(B)= 12 \biggl(2+\frac{4}{4}\biggl)-7\times 5=1.$$
	
	If at most two exterior edges of $B$ are incident to  faces of degree 4, then
	$$12f(B)-7e(B)\leq12 \biggl(2+\frac{2}{4}+\frac{2}{5}\biggl)-7\times 5=-\frac{1}{5}<0.$$

	\vskip3pt\noindent  {\bf Case 4.}
	$B=B_{5,a}$.

	If every exterior edge of $B$ is incident to  faces of degree at least 5, then  $$12f(B)-7e(B)\leq12 \biggl(3+\frac{5}{5}\biggl)-7\times 7=-1<0.$$
	
	If  $B$ has some exterior edges that are incident to  faces of degree 4, then by Lemma \ref{4no4}(3), three exterior edges of $B$ are  incident to a face of degree 4.
	By Lemma \ref{L:k4the62s}(3), the other two exterior edges of $B$ are incident to a face of degree at least 6. Thus  $$12f(B)-7e(B)\leq 12 \biggl(3+\frac{3}{4}+\frac{2}{6}\biggr)-7\times 7=0.$$
	
	\vskip3pt\noindent  {\bf Case 5.}
	$B=B_{5,b}$.

	Lemma \ref{L:k4the62s}(4) implies that $f(B)\leq 4+ \frac{4}{6}$.
	Hence $$12f(B)-7e(B)\leq12 \biggl(4+\frac{4}{6}\biggl)-7\times 8=0.$$

	We next construct a decomposition of $G$ by the triangular blocks of $\mathcal{B}(G)$.
	Note that $12f(B)-7e(B)>0$ if and only if $B=B_4$ and every exterior edge of $B$ is incident to a face of degree 4.
	By Lemma \ref{4no4}(2),  the only possible configuration is $B_{4}'$ as shown in Figure \ref{theta44}.
		We say two $B_{4}'$s are \textit{equivalent} if they share at least one exterior edge.
	Let $\mathcal{B}(G)=\{A_{1},A_{2},\dots,A_{t}\}$.
 If there exists a triangular block $A_i = B_4$ whose all exterior edges are incident to a face of degree 4, then by Lemma \ref{L:k4the62s}(2), we can get a subgraph $B_4'$, denoted by $A_i'$.
	For a subgraph $A_i'$, we can obtain a subgraph $A_i''$ such that $A_i' \subseteq A_i''$ and if there exists another $A_j'$ that is equivalent to a subgraph $A_k' \subseteq A_i''$, then $A_j' \subseteq A_i''$.
	We call such a subgraph $ A_i''$ a cluster.
	Let \(\mathcal{B}_1(G)\) be the collection of all clusters of $G$ and let \(\mathcal{B}_1'(G)\) be the subset of  \(\mathcal{B}(G)\)  that each of its members is contained in some subgraph in \(\mathcal{B}_{1}(G)\).
	Let \(\mathcal{B}_2(G) = \mathcal{B}(G) \setminus \mathcal{B}_1'(G)\). Then \(\mathcal{B}_1(G) \cup \mathcal{B}_2(G)\) is a decomposition of $G$.
	Denote \(\mathcal{B}_1(G) \cup \mathcal{B}_2(G) = \{T_1, T_2, \ldots, T_k\}\).

	For each $i \in \{1, 2, \ldots, k\}$, we shall prove that  $12f(T_i) - 7e(T_i)\leq 0$.
	If $T_i \in \mathcal{B}_2(G)$, by the discussions as above, we have $12f(T_i) - 7e(T_i) \leq 0$.
	If $T_i \in \mathcal{B}_1(G)$, then assume that $T_i$ contains $\ell $ copies of   $B_4'$, where $\ell \ge 1$.
	If $\ell =1$, then $T_i \cong B_4'$. Since \(12f(K_2) - 7e(K_2) \le -1\) and \(12f(B_4) - 7e(B_4) = 1\), we have $$12f(B_4') - 7e(B_4') \le -3 <0.$$
	If $\ell \ge 2$, then we contract every $B_4'$ of $T_i$ to a vertex. If  two $B_4'$s share $m_{js}$ common exterior edges, then we connect $v_j$ and $v_s$ with $m_{js}$ parallel edges. The resulting graph is denoted by $G'$.
	Since every $B_4'$ has four \begin{figure}[H]
		\centering
		\includegraphics[width=0.35\linewidth,
		height=0.4\textheight,
		keepaspectratio,
		draft=false]{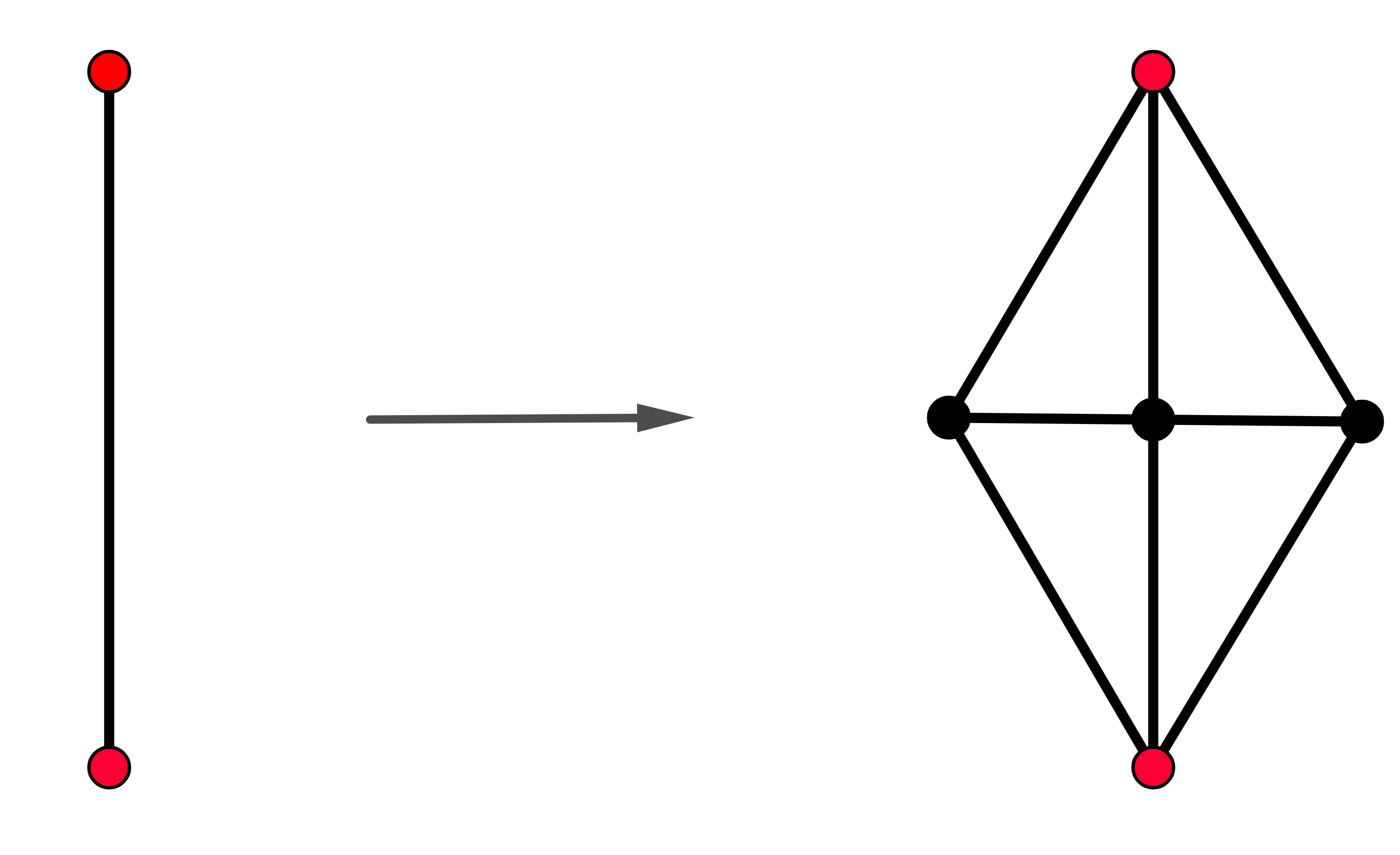}
		\caption{The replacement of an edge with a graph isomorphic to $B_{5,b}$.}	
		\label{b4k2}
	\end{figure}\noindent exterior edges, the maximum degree of $G'$ is at most 4.
	Let $a$ be the total number of edges that lie on the boundaries of  two distinct $B_{4}'$ and let $b$ be the number of the remaining edges.
	Then $2a+b=4\ell$ and
	$$2a=\sum_{v\in V(G')}d_{G'}(v)\leq4\ell.$$
	Hence $a\leq 2\ell$.
	Since \(12f(K_2) - 7e(K_2) \le -1\) and \(12f(B_4) - 7e(B_4) = 1\), we have
	$$12f(T_{i})-7e(T_{i})\leq \ell-(a+b)= -3\ell+a \leq -\ell\leq -1.$$
	
	Therefore, we have $e(G) \le \frac{12}{5}(n-2)$.
	  \end{proof}

	 \vspace{0.2cm}
	
	 	\noindent{\bf Proof of Theorem \ref{k462}.}
	 Let $G$ be an  $\{K_{4},\Theta_{6}^{2}\}$-free plane graph on $n\geq 6$ vertices.
	 The  inequality $e(G) \leq \frac{12}{5}(n-2)$ can be rewritten as $12n - 5e(G) \geq 24$.
	 We repeatedly do the operation of deleting vertices of degree at most 2 and  get an induced subgraph $G'$.  If $G'$ is not an
	 empty graph, then $\delta(G')\geq 3$.
	 Let $e_{i}$ be the number of edges deleted when deleting the $i$-th vertex where $i=1,2,\ldots,n-v(G').$
	 It is clear that  $e_{i}\leq 2$ and $12-5e_{i}\geq 2$.
	 Therefore,  we have
	\begin{align*}
		12n-5e(G) &= 12\bigl(v(G')+(n-v(G'))\bigr)-5\bigl(e(G')+e_{1}+e_{2}+\dots+e_{n-v(G')}\bigr)\\
		&= 12v(G')-5e(G')+\sum_{i=1}^{n-v(G')}(12-5e_{i})\\
		&\geq 12v(G')-5e(G')+2(n-v(G')).
	\end{align*}

	 If $v(G')=0$, then since  there exists at most one edge  between the last two vertices and $n\geq 6$, we have $$12n-5e(G)\geq 12\times 2-5\times 1+2(n-2)\geq 27.$$

	 If $v(G')\neq 0$, let  $G_{1}', G_{2}', \ldots ,G_{k}'$ be all components of $G'$ where $k\geq 1$.
	 Since $\delta(G')\geq 3$ and $G'$ is $K_{4}$-free, we have $v(G_{i}')\geq 5$ for every $i=1,2,\ldots,k$.
	 If $v(G_{i}')=5$, then since
	a plane triangulation on five 	 vertices with nine edges has  a subgraph isomorphic to $K_{4}$, we have $e(G_{i}')\leq 8$. Since $\delta(G')\geq 3$, we have $$3\times5\leq \sum_{v\in V(G_{i}')}d_{G_{i}'}(v)=2e(G_{i}'),$$ which implies $e(G_{i}')\geq \lceil 7.5 \rceil$.
	 Hence, if $v(G_{i}')=5$, then $e(G_{i}')=8$ and $12v(G_{i}')-5e(G_{i}')=20$.
	 If $v(G_{i}')\geq 6$, then by Lemma \ref{L:K4theta62}, we have
	  $12v(G_{i}')-5e(G_{i}')\geq 24$.
	 Therefore, if there  exists a component $G_{i}'$ on at least six vertices, then  $12n-5e(G)\geq 24.$	
	 We next consider the case that all components are  on five vertices.
	 Then  $$12n-5e(G)\geq \sum_{i=1}^{k}(12v(G_{i}')-5e(G_{i}'))+2(n-v(G'))\geq 20k+2(n-v(G')).$$	
	 If $k\geq 2$ or $n\geq v(G')+2$, then $12n-5e(G)\geq 24$.
	 If $k=1$ and $n\leq v(G')+1$, then $v(G')=5$, $n=6$ and $G'$ is isomorphic to $B_{5,b}$.
	 Since $G$ is $\Theta_{6}^{2}$-free, we know that  the last vertex deleted, denoted as $v$, cannot have degree 2 in $G$.
	Then $e(G)\leq e(G')+1=9$ and
	 $$
	 	12n-5e(G)\geq 12\times 6-5\times 9\geq 27.
	 $$
	 In all the cases, we have $12n-5e(G)\geq24$.
	 Hence $ex_{_\mathcal{P}}(n,\{K_{4},\Theta_{6}^{2}\})\leq \frac{12}{5}(n-2)$ for  $n\geq 6$.

		From the proof of Lemma \ref{L:K4theta62}, we see that if $\delta(G)\geq 3$, then the equality $e(G)=\frac{12}{5}(n-2)$ is achieved if  every non-trivial triangular block of $G$  is a  $B_{5,b}$ and every exterior edge of  $B_{5,b}$ is incident to a face of degree 6.
		We first construct a family of extremal  $\{K_{4},\Theta_{6}^{2}\}$-free plane  graphs for $n\geq12$ and $n\equiv2~(\text{mod}~10)$.
		Let $n=10k+2$ where $k\geq1$.
		Let $P_{k}$ denote a path on $k$ vertices.
		Let $H=P_{k}+P_{2}$ for $k\geq1$.
		Then $H$ is a plane  triangulation graph on $k+2$ vertices with $3k$ edges.
		Replace every edge of $H$ by a graph isomorphic to $B_{5,b} $ as is illustrated in Figure \ref{b4k2} and denote the resulting graph by $H^{*}$.      Since $H$ is       a plane  triangulation    and the face set  $F(H^{*})  $ consists of faces of degree 3 (every being an inner face  of $B_{5,b}$) and 6.
	Since any two triangular  blocks $B_{5,b} $ are edge-disjoint,   $H^{*}$ is      $\{K_{4},\Theta_{6}^{2}\}$-free.
		From the construction of $H^{*}$, we have $e(H^{*})  =8e(H) =24k$ and $v(H^{*})  =v(H)+3e(H) =10k+2$.
	Therefore $e(H^{*})=\frac{12(v(H^{*})-2)}{5}$, that is, $H^{*}$ is an extremal $\{K_{4},\Theta_{6}^{2}\}$-free graph on $n \geq 12$ vertices   and $n \equiv 2 \pmod{10}$.

		We now construct a different type of  extremal $\{K_{4},\Theta_{6}^{2}\}$-free plane  graphs for $n\equiv42\ (\mathrm{mod}\ 100)$.
	Let $G_0'$ be the graph obtained from the graph as shown in Figure \ref{figk4t6}(a) by deleting  the dashed edges. Then $G_0'$ has $40$ vertices and $80$ edges.
		For $k \ge 1$, we construct $G_k'$ by embedding $G'_{k-1}$ into the blue central such that the  exterior edges of $G'_{k-1}$ coincide with the blue edges as shown in Figure \ref{figk4t6}$(b)$. A graph
$G_{k}$ is constructed from $G_{k}'$ as follows: (1) add four edges to partition every face of degree 20 into four pentagonal faces and one quadrilateral face; (2) add one new vertex and connect it by four edges to the four vertices of the newly formed quadrilateral face, as indicated by the dashed edges in Figure \ref{figk4t6}$(a)$.
		 We can see that all cycles of length four are contained in the subgraph that is isomorphic to $B_{5,b}$ and no two cycles of length four share exactly one common edge.
		 Moreover, for every edge of a $K_3$, the other face incident to that edge is  of degree 3 or 6.
		Then $G_{k}$ is an extremal $\{K_{4},\Theta_{6}^{2}\}$-free plane graph on $100k+42$ vertices with $240k+96$ edges for $k\geq0$. 	
		This completes the proof of Theorem \ref{k462}.
		$\hfill\blacksquare$

	\begin{figure}[H]
		\centering
		\begin{minipage}{0.39\linewidth}
			\centering
			\includegraphics[width=\linewidth, keepaspectratio]{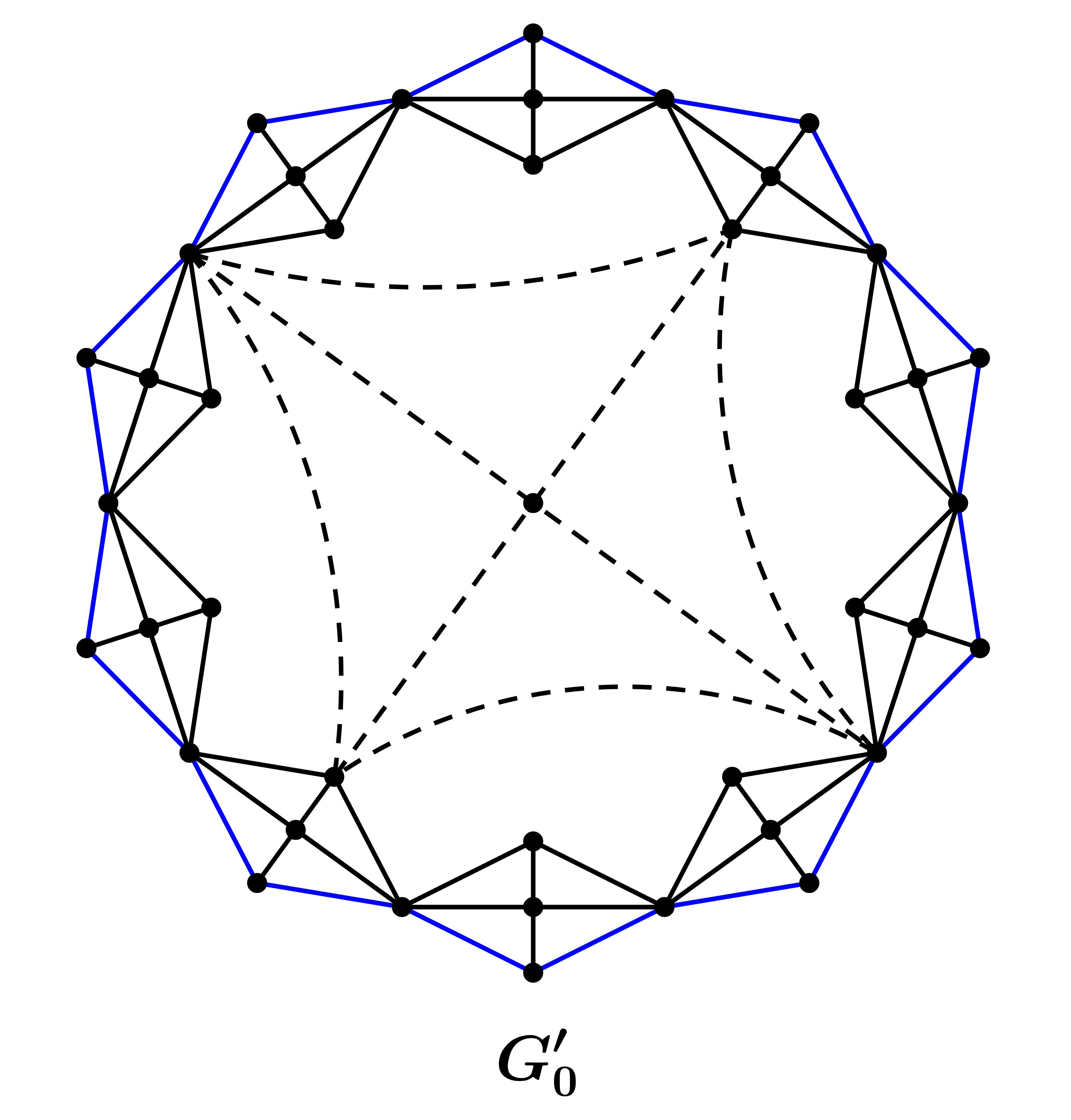}
			\large $(a)$
		\end{minipage}
		\hspace{0.05\linewidth}
		\begin{minipage}{0.36\linewidth}
			\centering
			\includegraphics[width=\linewidth, keepaspectratio]{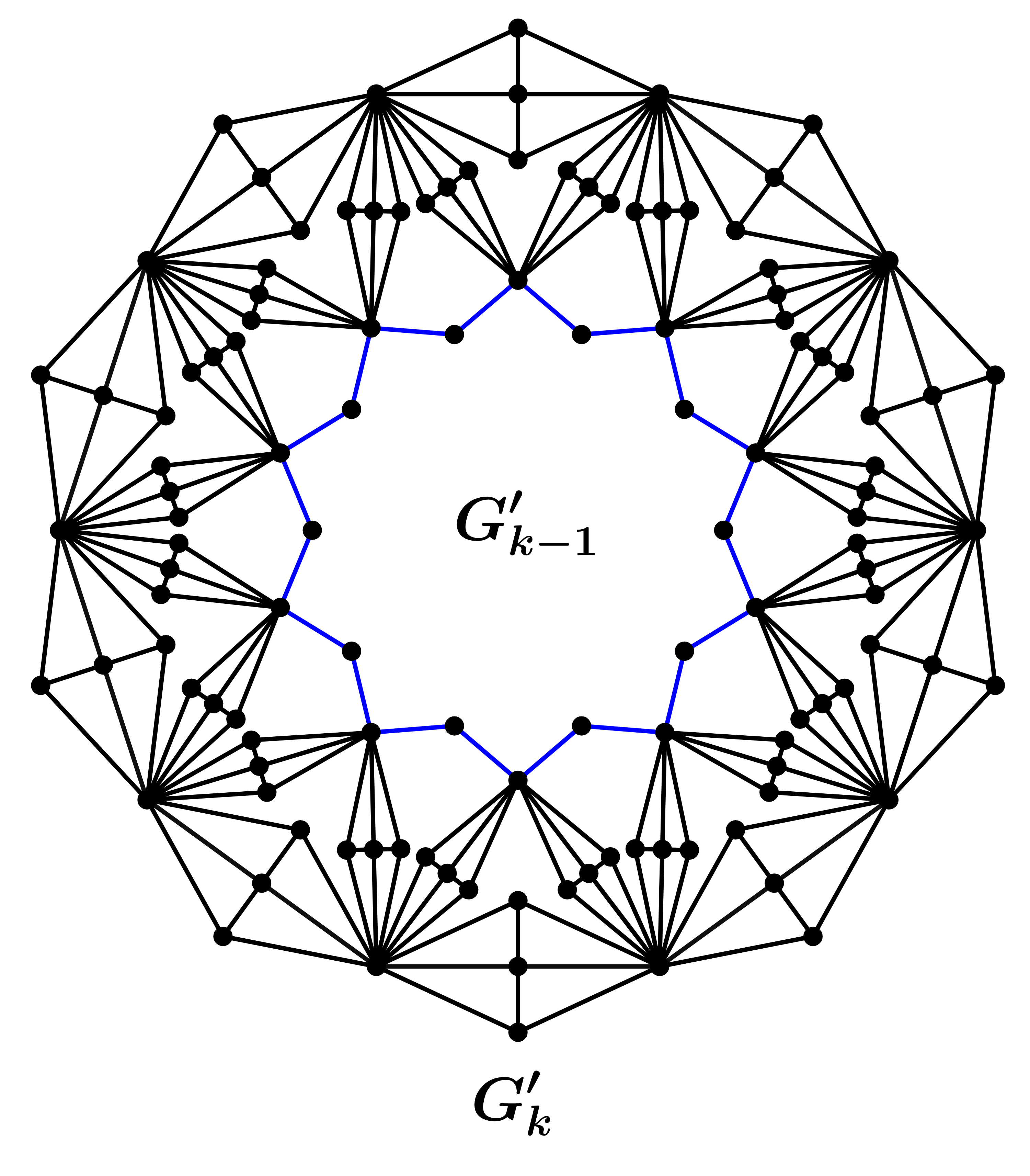}
			\large $(b)$
		\end{minipage}
		\caption{Extremal $\{K_{4},\Theta_{6}^{2}\}$-free plane graphs.}
		\label{figk4t6}
	\end{figure}

	\vskip3pt
	
	\noindent {\bf\large Conflicts of  interest}
	
	The authors declare that they have no conflict interest.

	\vskip10pt
\noindent{\bf\large Data availability}

No data was used for the research described in the article.

	\vskip10pt
	\noindent {\bf\large Acknowledgements}
	
	This work was supported by the National Natural Science Foundation of China (No.
	12271311) and the Natural Science Foundation of Hebei Province, China (No.
	A2025202034).

	\end{document}